\documentclass[10pt,a4paper,reqno]{amsart} 

\usepackage[english]{babel}
\usepackage[utf8]{inputenc}
\usepackage{rotating}
\usepackage{amssymb,amsfonts,amsmath,stmaryrd}
\usepackage[T1]{fontenc}

\usepackage[normalem]{ ulem }
\usepackage{hyperref}

\newcommand{\Cgf}{{\sf C}}

\newcommand{\Ggf}{{\sf G}}
\newcommand{\bGgf}{\overline{\Ggf}}
\newcommand{\Hgf}{{\sf H}}

\newcommand{\Qgf}{{\sf Q}}

\newcommand{\Rgf}{{\sf R}}
\newcommand{\Sgf}{{\sf S}}
\newcommand{\Tgf}{{\sf T}}

\newcommand{\Wgf}{{\sf W}}
\newcommand{\SL}{\mathit{SL}}
\newcommand{\PSL}{\mathit{PSL}}

\newtheorem{Theorem}{Theorem}[section]
\newtheorem{Lemma}[Theorem]{Lemma}
\newtheorem{Proposition}[Theorem]{Proposition}

\renewcommand\th{\vartheta}

\newcommand{\beq}{\begin{equation}}
\newcommand{\eeq}{\end{equation}}
\def\emm#1,{{\em #1}}

\catcode`\@=11
\def\section{\@startsection{section}{1}%
 \z@{.7\linespacing\@plus\linespacing}{.5\linespacing}%
 {\normalfont\bfseries\scshape\centering}}

\def\subsection{\@startsection{subsection}{2}%
  \z@{.5\linespacing\@plus\linespacing}{.5\linespacing}%
  {\normalfont\bfseries\scshape}}

\def\subsubsection{\@startsection{subsubsection}{3}%
 \z@{.5\linespacing\@plus\linespacing}{-.5em}
 {\normalfont\bfseries}}
\catcode`\@=12

\newcommand{\zs}{\mathbb{Z}}

\newcommand{\om}{\omega} 
\renewcommand{\epsilon}{\varepsilon}

\graphicspath{{Figures/}}

\usepackage{tikz}\usetikzlibrary{graphs,quotes,fit,positioning,matrix,calc,decorations.markings,angles,decorations.pathmorphing,decorations.pathreplacing}
\title{The six-vertex model on random planar maps revisited}

\DeclareMathOperator{\tr}{tr}
\newcommand\ZZ{\mathbb{Z}}

\long\def\remint#1#2#3{\begin{tikzpicture}[baseline=-\the\dimexpr\fontdimen22\textfont2\relax]\node[outer sep=0pt,draw=black,fill=#1,fill opacity=0.5,text opacity=1,rectangle,rounded corners,text width=#3,align=flush left]{#2};
\end{tikzpicture}}
\newdimen\mylinewidth
\newcommand\rem[2][cyan!40!green]{%
\mylinewidth=.97\linewidth
\setbox0\hbox{#2}%
\ifdim\wd0>\mylinewidth
\strut\\\vbox{%
  \hsize=\linewidth
  \kern-\lineskip
  \raggedright
  \strut\remint{#1}{#2}{\mylinewidth}%
}%
\else
\hskip0pt plus\linewidth
\penalty0
\,\remint{#1}{#2}{\wd0}\,
\hskip0pt plus\linewidth
\fi%
}
\renewcommand{\d}{\mathrm{d}}
\newcommand{\der}{\partial}
\newcommand\W{W^{(0)}}
\newcommand\G{G^{(0)}}
\newcommand\bG{\bar G^{(0)}}
\renewcommand\H{H^{(0)}}

\begin{document}

\author{
Andrew Elvey Price%
	\and
Paul Zinn-Justin
}

\thanks{Andrew Elvey Price was supported by the European Research Council (ERC) in the European Union’s Horizon 2020 research and innovation programme, under the Grant Agreement No.~759702. 
PZJ was supported by ARC grants FT150100232 and DP180100860.}

\newcommand{\Addresses}{{
  \bigskip
  \footnotesize

  A.~Elvey Price, \textsc{Laboratoire Bordelais de Recherche en Informatique, UMR 5800, Universit\'e de Bordeaux, 351 Cours de la Libération, 33405 Talence Cedex, France}\par\nopagebreak
  \textit{E-mail address:} \texttt{andrew.elvey@univ-tours.fr}

  \medskip

  P.~Zinn-Justin, \textsc{School of Mathematics and Statistics, The University of Melbourne, Victoria 3010, Australia}\par\nopagebreak
  \textit{E-mail address:} \texttt{pzinn@unimelb.edu.au}

}}

\maketitle
\begin{abstract}
We address the six vertex model on a random lattice, which in combinatorial terms corresponds to the enumeration of weighted 4-valent planar maps equipped with an Eulerian orientation. This problem was exactly, albeit non-rigorously solved by Ivan Kostov in 2000 using matrix integral techniques. We convert Kostov's work to a combinatorial argument involving functional equations coming from recursive decompositions of the maps, which we solve rigorously using complex analysis. We then investigate modular properties of the solution, which lead to simplifications in certain special cases. In particular, in two special cases of combinatorial interest we rederive the formulae discovered by Bousquet-M\'elou and the first author.
\end{abstract}
\section{Introduction}
\begin{figure}[ht]
\centering
   \includegraphics[scale=0.7]{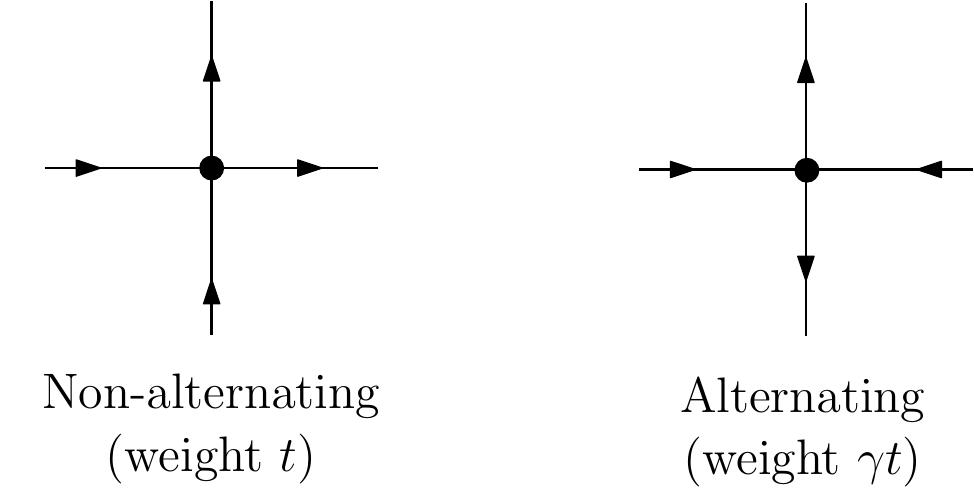} 
   \caption{The two types of vertices in the six vertex model on a random lattice.}
   \label{fig:two vertex types}
\end{figure}
We determine the generating function $\Qgf(t,\gamma)$ which counts these weighted 4-valent planar maps equipped with an Eulerian orientation with $t$ counting vertices and $\gamma$ counting {\em alternating} vertices, that is vertices in which the two outgoing edges are opposite each other (see figure \ref{fig:two vertex types}). This is equivalent to the six-vertex model on dynamical random lattices, which was exactly, but non-rigorously solved by Kostov in 2000 \cite{Kostov-6v} using matrix integral techniques, after work by the second author \cite{artic10}. In this article we give a purely combinatorial rephrasing of Kostov's arguments, yielding a (rigorous) exact solution. In doing so, we correct a mistake, yielding a much simplified expression for $\Qgf(t,\gamma)$ compared to the (incorrect) expression that could be extracted directly from \cite{Kostov-6v}. Our solution is in terms of the classical Jacobi theta function $\th\equiv \theta_{1}$
\begin{align}\label{theta-def0}
\th(z)\equiv\th(z,q)&= 2\sin(z) q^{1/8} \prod_{n=1}^{\infty} (1-2\cos(2z)q^n+q^{2n})(1-q^n),\\
&= -i \sum_{n\in \zs} (-1)^n e^{(n+1/2)^2\pi i\tau+(2n+1)iz},
\end{align}
where $q=e^{2\pi i\tau}$. We will alternately see this as a power series in $q$ or an analytic function of $\tau$ and $z$. Our main theorem is the following:
\begin{Theorem}\label{thm:allgamma}
Write $\gamma=-2\cos(2\alpha)$, and   let $q(t,\gamma)\equiv q= t+ \left( 
6\,\gamma+6 \right) {t}^{2}+\cdots $ be the
unique formal power series in $t$ with constant
  term $0$ satisfying
\[
t=  \frac{\cos\alpha}{64\sin^3\alpha}
\left(
-\frac{\th(\alpha,q)\th'''(\alpha,q)}{\th'(\alpha,q)^2}+\frac{\th''(\alpha,q)}{\th'(\alpha,q)}
\right),
\]
where all derivatives are with respect to the first variable. Moreover, define the series $\Rgf(t,\gamma)$ by
\[
\Rgf(t,\gamma)=\frac{\cos^2\alpha}{96\sin^4\alpha}
\frac{\th(\alpha,q)^2}{\th'(\alpha,q)^2}
\left(-\frac{\th'''(\alpha,q)}{\th'(\alpha,q)}
+\frac{\th'''(0,q)}{\th'(0,q)}\right).
\]
 Then the generating function of quartic rooted planar Eulerian orientations, counted by
vertices, with a weight $\gamma$ per alternating vertex is
\[
\Qgf(t,\gamma)= \frac{1}{(\gamma+2)t^2}\left( t-(\gamma+2)t^2-\Rgf(t,\gamma)\right).
\]
\end{Theorem}

Recently this problem has been considered in the combinatorics literature \cite{BoBoDoPe,BM_AEP,elvey-guttmann17}. In particular Bousquet-M\'elou and the first author \cite{BM_AEP} exactly solved the unweighted case $\gamma=1$ as well as the case $\gamma=0$, which they showed, bijectively, to be equivalent to the enumeration of Eulerian orientations. Their solution has a much simpler form than our solution for general $\gamma$, in that the series $\Rgf(t,\gamma)$ is the functional inverse of a simple hypergeometric series in these cases. In the final sections, we show how to derive this solution from our more general solution, and we find that similar simplifications occur whenever $\gamma$ is of the form $2\cos(\pi k)$, for $k\in\mathbb{Q}$. We note that \cite{BM_AEP} was combined with an early version of the present work in the form of an extended abstract \cite{artic74}.

The bijection used in \cite{BM_AEP} shows, more generally, that the generating function $\Qgf(t,\gamma)$ counts planar maps equipped with an Eulerian {\em partial} orientation in which each edge may or may not be directed, and each vertex has equally many incoming as outgoing edges. In this context $\Qgf(t,\gamma)$ counts these partial orientations by edges ($t$) with a weight $\gamma$ per undirected edge.

As noted in \cite{BM_AEP}, Kostov's work was largely overlooked in the combinatorics literature due in part to the ``unfamiliar language and techniques used''. One of our aims in this work is to describe the techniques used in more combinatorial language, so that they can become more familiar in combinatorics. Indeed, we expect that similar techniques could be applied to a wide variety of functional equations appearing in combinatorics. As an example, the first author has adapted these techiques to the enumeration of walks on various lattices by winding number (see \cite{AEP_winding} for an extended abstract). More generally we note that there is a strong similarity to the enumeration of certain lattice walks confined to a quadrant in terms of the Weierstrass elliptic function \cite{Rachel_Kurkova, Fayolle_Raschel, BeBoRa}.

The outline of this article is as follows: In Section \ref{sec:tutte}, we derive a system of functional equations which characterises the generating function $\Qgf(t,\gamma)$. In Section \ref{sec:derivation} we solve these equations under an assumption known generally as the {\em one-cut assumption}, thereby non-rigorously deriving Theorem \ref{thm:allgamma}. These two sections are essentially a rephrasing (and correction) of Kostov's work \cite{Kostov-6v}. We describe the relationship between our functional equations and the matrix integral approach \cite{Kostov-6v,artic10} in Appendix \ref{sec:matrix}. In Section \ref{sec:proof}, we use our non-rigorously derived solution to find the unique series satisfying the functional equations of section \ref{sec:tutte}, thereby {\em proving} Theorem \ref{thm:allgamma}. Sections \ref{sec:DE}, \ref{sec:modular_properties} and \ref{sec:particular_cases} are dedicated to analysing the series in Theorem \ref{thm:allgamma}. In particular, in Section \ref{sec:DE} we derive a differential equation relating the series $\Qgf(t,\gamma)$ to simpler series. In Section \ref{sec:modular_properties}, we use this differential equation to derive modular properties of the solution for infinitely many values of $\gamma$. Finally, in Section \ref{sec:particular_cases}, we analyse certain the solution for certain specific such values of $\gamma$, including $\gamma=0$ and $\gamma=1$, in which cases we rederive the solutions of \cite{BM_AEP}.

\section{Functional equations for $\Qgf(t,\gamma)$}\label{sec:tutte}
In this section we derive functional equations which characterise $\Qgf(t,\gamma)$ using the recursive method \`a la Tutte \cite{tutte-census-maps}. We describe the connection with the matrix model studied by Kostov and the second author \cite{Kostov-6v,artic10} in Appendix \ref{sec:matrix}. In particular, this section (and the start of the next section) can be seen as a combinatorial rephrasing of the matrix integral approach, after which we arrive at the same functional equations as Kostov.


We start with a bijection relating the generating function $\Qgf(t,\gamma)$ to a generating function counting partially oriented 3-valent maps in which each vertex is incident to one incoming edge, one outgoing edge and one undirected edge. In these maps there are two types of vertices, shown in Figure \ref{fig:two_A_vertex_types}, which we call {\em right turn vertices} and {\em left turn vertices}. Each right turn vertex is given weight $\omega^{-1}$, each left turn vertex is given weight $\omega$ and each undirected edge is given weight $t$. As usual, there is a unique root edge and incident root vertex such that the root edge is oriented away from the root vertex. The sum of the resulting weights of all such maps is denoted by the generating function $\Cgf(t,\omega)$. We will now prove the following lemma, relating $\Cgf(t,\omega)$ and $\Qgf(t,\gamma)$:
\begin{Lemma}
The generating functions $\Qgf(t,\gamma)$ and $\Cgf(t,\omega)$ are related by the equation
\[\Qgf(t,\omega^{2}+\omega^{-2})=\Cgf(t,\omega)\]
\end{Lemma}
\begin{proof}
Starting with a 4-valent Eulerian orientation, we split each vertex into a pair of three valent vertices as shown in Figure \ref{fig:two vertex types transformations} such that each resulting vertex is incident to one incoming edge, one outgoing edge and one undirected edge. For alternating vertices there are two possible choices, one with weight $\omega^2 t$ and one with weight $\omega^{-2} t$, while there is only one choice for non-alternating vertices, and it has weight $t$. This explains the correspondence $\gamma=\omega^2+\omega^{-2}$, as this ensures that the total weight of the possible pairs of vertices 3-valent vertices is equal to the weight of the original 4-valent vertex in each case. Finally note that we can reverse this transformation by simply contracting all undirected edges.
\end{proof}
\begin{figure}[ht]
\centering
   \includegraphics[scale=0.7]{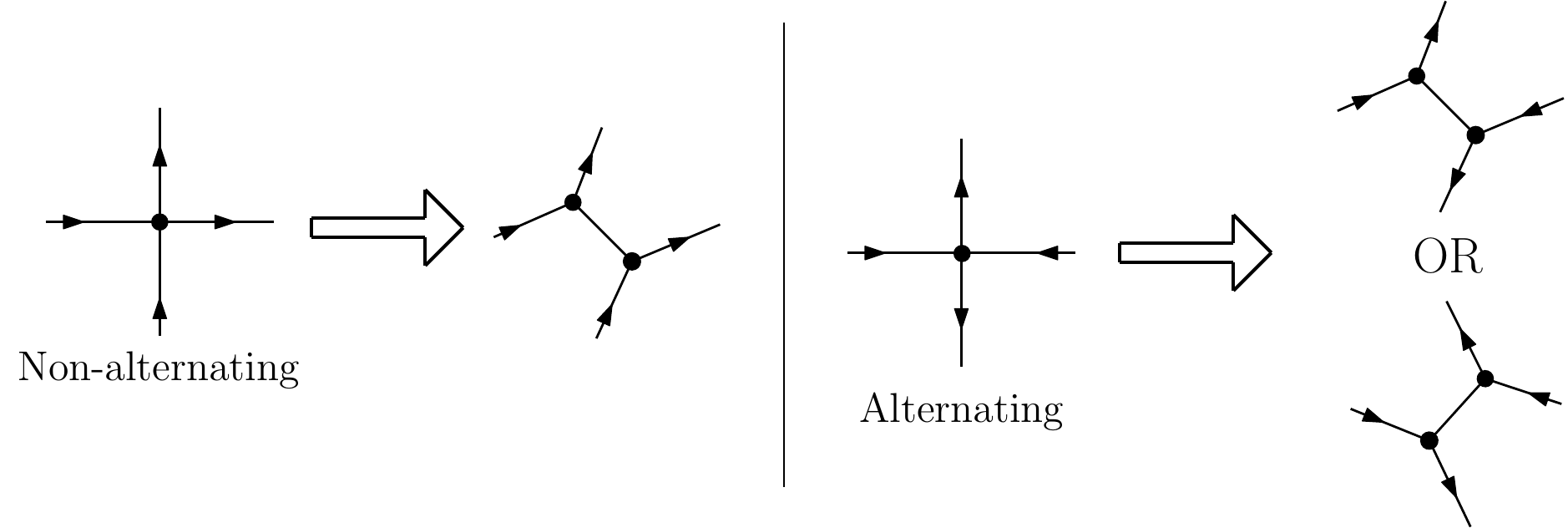} 
   \caption{Transforming degree 4 vertices to pairs of degree 3 vertices.}
   \label{fig:two vertex types transformations}
\end{figure}

\begin{figure}[ht]
\centering
   \includegraphics[scale=1]{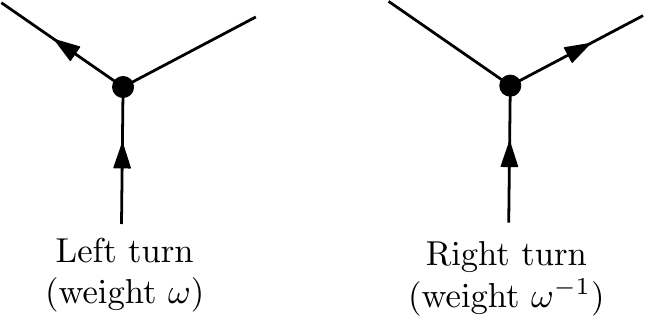} 
   \caption{The two vertex types allowed in the cubic partial Eulerian orientations.}
   \label{fig:two_A_vertex_types}
\end{figure}

In order to characterise the generating function $\Cgf(t,\omega)$ as the solution to a system of functional equations, we introduce two new series $\Wgf(x)\equiv\Wgf(t,\omega,x)$ and $\Hgf(x,y)\equiv\Hgf(t,\omega,x,y)$. Each of these series counts planar partial Eulerian orientations in which each non-root vertex is either a right-turn vertex or a left-turn vertex, the only difference being the weight and allowed types of the root vertex. As usual, each right turn vertex (resp. left turn vertex) is given weight $\omega^{-1}$ (resp. $\omega$) and each undirected edge is given weight $t$. For $\Wgf(x)$, the root vertex may only be adjacent to undirected edges, and the weight of this vertex is $x^{k}$, where $k$ is its degree. For $H(x,y)$ the root vertex has exactly one outgoing edge, which is the root edge, and exactly one incoming edge. Each incidence between an undirected edge and the root vertex is either on the {\em left} or {\em right} of these edges. The weight of the root vertex is $x^{j}y^{k}$, where $j$ is the number of incidences on the right of the two directed edges and $k$ is the number of incidences on the left of these two edges. We call maps counted by $\Wgf(x)$ {\em W-maps} and we  call maps counted by $\Hgf(x)$ {\em H-maps}.

The series $\Hgf(t,\omega,0,0)\equiv\Hgf(0,0)$ then counts H-maps in which the root vertex is incident to only two edges, one outgoing and one incoming. replacing these two edges with a single (root) edge yields a C-map, hence $\Hgf(t,\omega,0,0)\equiv\Hgf(0,0)=\Cgf(t,\omega)$

\begin{figure}[ht]
\centering
   \includegraphics[scale=1.2]{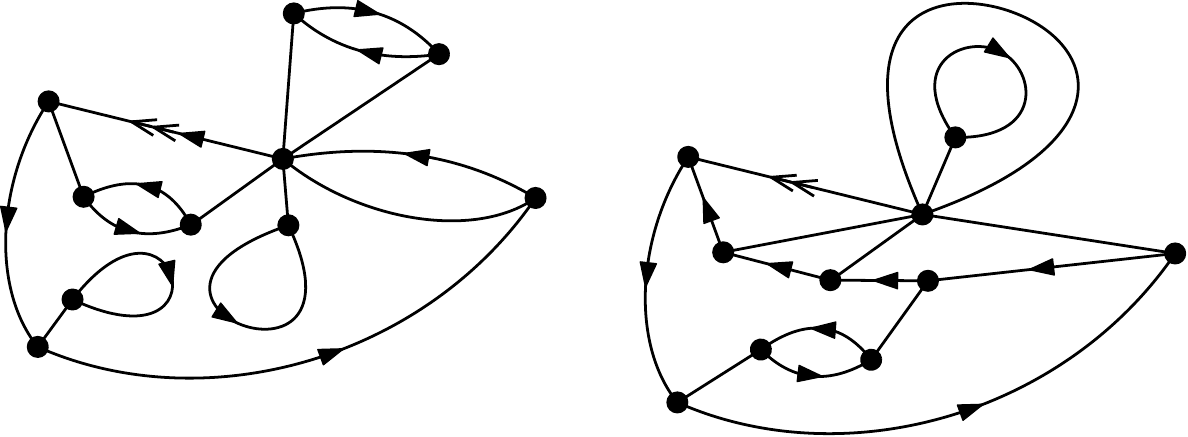} 
   \caption{Left: a map which contributes $\omega^{-3}x^{2}y^{3}t^{7}$ to $\Hgf(x,y)$. Right: a map which contributes $\omega^{3}x^{7}t^{8}$ to $\Wgf(x)$. The root edges are marked by double arrows.}
   \label{fig:HandW_examples}
\end{figure}

In forthcoming lemmas, we will show that these series are characterised by the following two equations:
\begin{equation}\label{wloopcombi}\Wgf(x)=x^2t\Wgf(x)^{2}+\omega xt\Hgf(x,0)+\omega^{-1}xt\Hgf(0,x)+1\end{equation}
\begin{equation}\label{hloopcombi}\Hgf(x,y)=\Wgf(x)\Wgf(y)+\frac{\omega^{-1}}{y}\left(\Hgf(x,y)-\Hgf(x,0)\right)+\frac{\omega}{x}\left(\Hgf(x,y)-\Hgf(0,y)\right).\end{equation}

In particular, the main series of interest $\Cgf(t,\omega)$ is related to these by the equation
\[\Cgf(t,\omega)=\Hgf(0,0)=\frac{[x^1]\Wgf(x)}{t(\omega+\omega^{-1})}.\]

\begin{Lemma}\label{lem:uniqueness} Equations \eqref{wloopcombi} and \eqref{hloopcombi} define unique series $\Hgf(x,y)$ and $\Wgf(x)$ in $\mathbb{C}[x,y][[t]]$ and $\mathbb{C}[x][[t]]$, respectively.\end{Lemma}
\begin{proof}
Suppose the contrary and let $\Wgf_{1}(x),\Hgf_{1}(x,y)$ and $\Wgf_{2}(x),\Hgf_{2}(x,y)$ be distinct pairs of series which solve the equations Then let $k$ be minimal such that either
\[[t^{k}]\Wgf_{1}(x)\neq [t^{k}]\Wgf_{2}(x)~~~~~~\text{ or }~~~~~~[t^{k}]\Hgf_{1}(x,y)\neq [t^{k}]\Hgf_{2}(x,y).\]
For $j<k$, we have $[t^{j}]\Wgf_{1}(x)=[t^{j}] \Wgf_{2}(x)$ and $[t^{j}]\Hgf_{1}(x,y)=[t^{j}]\Hgf_{2}(x,y)$, from which it follows from \eqref{wloopcombi} that $[t^{k}]\Wgf_{1}(x)=[t^{k}]\Wgf_{2}(x)$. Hence, $[t^{k}]\Hgf_{1}(x,y)\neq [t^{k}]\Hgf_{2}(x,y)$. Let $r,s\in\mathbb{C}$ satisfy $[t^{k}x^{r}y^{s}](\Hgf_{1}(x,y)-\Hgf_{2}(x,y))\neq0$ such that $r+s$ is maximal. Then the $[t^{k}x^{r}y^{s}]$ coefficient of the right hand side of \eqref{hloopcombi} is equal for both solutions $\Wgf_{1}(x),\Hgf_{1}(x,y)$ and $\Wgf_{2}(x),\Hgf_{2}(x,y)$, so $[t^{k}x^{r}y^{s}](\Hgf_{1}(x,y)-\Hgf_{2}(x,y))=0$, a contradiction.
\end{proof}

\begin{figure}[ht]
\centering
   \includegraphics[scale=1.2]{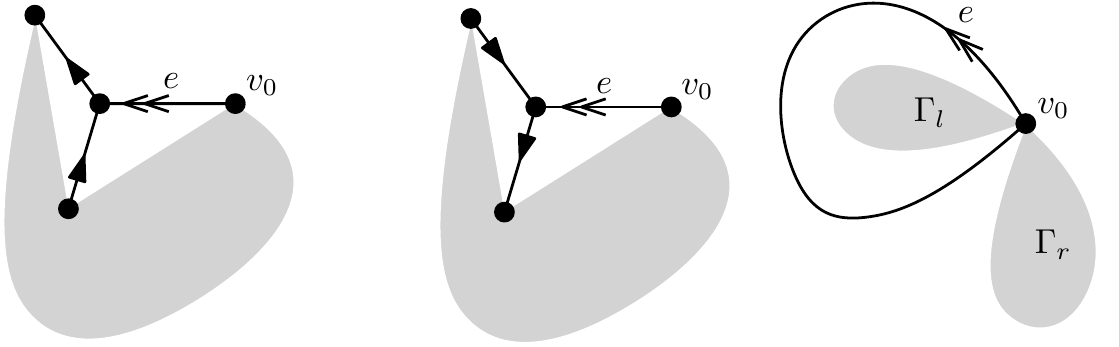} 
   \caption{The three types of W-maps apart from the atomic map used in the proof of Lemma \ref{lem:Wequation}. }
   \label{fig:Wproofcases}
\end{figure}

Now we will give combinatorial proofs of \eqref{wloopcombi} and \eqref{hloopcombi}.
\begin{Lemma}\label{lem:Wequation}The series $\Hgf(x,y)$ and $\Wgf(x)$ satisfy the equation
\[\Wgf(x)=x^2t\Wgf(x)^{2}+\omega xt\Hgf(x,0)+\omega^{-1}xt\Hgf(0,x)+1.\]\end{Lemma}
\begin{proof}
Consider maps $\Gamma$ which contribute to $\Wgf(x)$. The situation where $\Gamma$ is a single vertex contributes $1$ to $\Wgf(x)$. Otherwise, we will consider three cases, illustrated in Figure \ref{fig:Wproofcases}. In each case, let $v_{0}$ be the root vertex and let $e$ be the root edge. In the first case,  the other end of the root edge is a left-turn vertex. Contracting the root edge yields a H-map in which the root vertex $v_{0}$ has no incidences on the left of $e$, that is, map which contributes to $\Hgf(x,0)$. Hence this case contributes $\omega xt\Hgf(x,0)$. Similarly, the case where the other end of the root vertex is a right turn vertex contributes $\omega^{-1} xt\Hgf(0,x)$. In the remaining case, where both ends of the root edge are attached to the root vertex, the map splits into two pieces $\Gamma_{r}$ (right of the root edge), and $\Gamma_{l}$ (left of the root edge). The two maps $\Gamma_{l}$ and $\Gamma_{r}$ can be any pair of maps counted by $\Wgf(x)$, Hence this case contributes $x^2t\Wgf(x)^2$.

Adding the contributions from all four cases yields the desired equation
\end{proof}
\begin{figure}[ht]
\centering
   \includegraphics[scale=1.2]{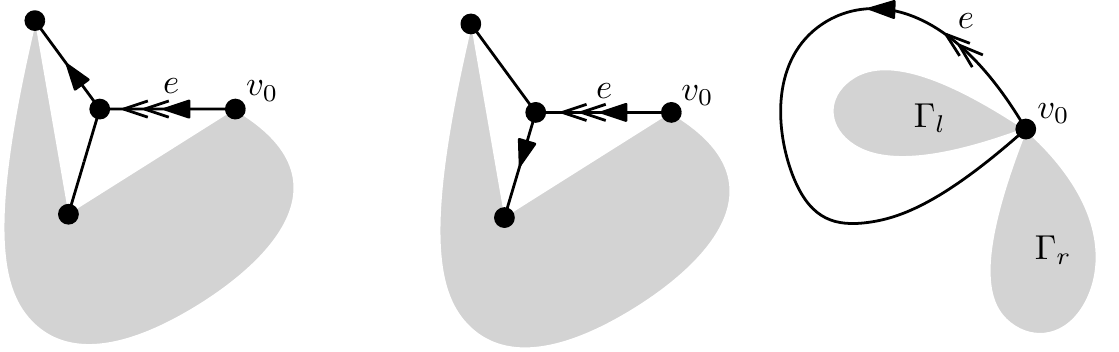} 
   \caption{The three types of H-maps used in the proof of Lemma \ref{lem:Hequation}. }
   \label{fig:Hproofcases}
\end{figure}

\begin{Lemma}\label{lem:Hequation}The series $\Hgf(x,y)$ and $\Wgf(x)$ satisfy the equation
\[\Hgf(x,y)=\Wgf(x)\Wgf(y)+\frac{\omega^{-1}}{y}\left(\Hgf(x,y)-\Hgf(x,0)\right)+\frac{\omega}{x}\left(\Hgf(x,y)-\Hgf(0,y)\right).\]\end{Lemma}
\begin{proof}
Consider maps $\Gamma$ which contribute to $\Hgf(x,y)$. We will consider three cases, illustrated in Figure \ref{fig:Hproofcases}. In each case, let $v_{0}$ be the root vertex and let $e$ be the root edge. In the first case, the other end of the root edge is a right-turn vertex. Contracting the root edge yield a H-map $\Gamma_{1}$ in which the root vertex has one more incidence on the left of $e$ than in $\Gamma$. The map $\Gamma
_{1}$ can be any map contributing to $H$ with at least one incidence to the root vertex on the left of $e$ and such a map $\Gamma_{1}$ uniquely determines $\Gamma$. Hence this case contributes $\frac{\omega}{x}(\Hgf(x,y)-\Hgf(0,y))$ to $\Hgf(x,y)$. Similarly, the second case, in which the other end of the root vertex is a right-turn vertex contributes $\frac{\omega^{-1}}{y}(\Hgf(x,y)-\Hgf(x,0))$. 

In the remaining case, where both ends of the root edge are attached to the root vertex, the map splits into two W-maps $\Gamma_{r}$ (right of the root edge), and $\Gamma_{l}$ (left of the root edge). If $k_{l}$ and $k_{r}$ are the degrees of the root vertex in $\Gamma_{l}$ and $\Gamma_{r}$, respectively, then the weight of the root vertex of $\Gamma$ is $x^{k_{r}}y^{k_{l}}$, hence this case contributes $\Wgf(x)\Wgf(y)$.

Adding the contributions from all four cases yields the desired equation.
\end{proof}

\section{Derivation of Theorem \ref{thm:allgamma} making the one-cut assumption}\label{sec:derivation}
In this section we solve the equations \eqref{wloopcombi} and \eqref{hloopcombi} derived in the previous section while making an assumption known as the ``one cut assumption'' so it should not be seen as rigorous. Although we believe that it is possible to prove this assumption directly, it is simpler to prove that our solution satisfies equations \eqref{wloopcombi} and \eqref{hloopcombi}, and is therefore the unique solution described by Lemma \ref{lem:uniqueness}. We give this proof in the following section. This section closely follows the method Kostov \cite{Kostov-6v} but we fix a mistake in one of his integrals.

As a first step, we fix $\omega\in\mathbb{C}$ to satisfy $|\omega|=1$ and $t\in\mathbb{R}$ to be positive but small ($t\in[0,\frac{1}{64}]$ is sufficient). We define four complex analytics functions $\W(x)$, $\G(x)$, $\bG(x)$ and $\H(x,y)$ of $x$ (and $y$) as follows:
\begin{align*}\W(x)&=\frac{1}{x}\Wgf\left(\frac{1}{x}\right)&\H(x,y)&=\frac{1}{xy}\Hgf\left(\frac{1}{x},\frac{1}{y}\right)\\
\G(x)&=\frac{1}{x}\Hgf\left(\frac{1}{x},0\right)&
\bG(x)&=\frac{1}{x}\Hgf\left(0,\frac{1}{x}\right).
\end{align*}
In particular, using simple bounds, these can be seen to converge for sufficiently large $x$ and $y$. In terms of these functions, equations \eqref{wloopcombi} and \eqref{hloopcombi} become:
\begin{align}
\label{eq:loop1c}
0&=\W(x)^2-t^{-1}(x \W(x)-1)+\omega \G(x)+\omega^{-1}\bG(x),
\\\label{eq:loop2c}
0&=\W(x)\W(y)+\H(x,y)(-1+\omega x+\omega^{-1}y)-\omega\bG(y)-\omega^{-1}\G(x).
\end{align}


\textbf{Assumption (the ``one-cut assumption''):} The
function $\W(x)$ is analytic in $x$ except on a single cut $[x_1,
x_2]$ on the positive real line. Note that $x_1,x_2$ depend on $t$, but the dependence is suppressed in the notation.

We apply this assumption immediately
by taking the difference of \eqref{eq:loop1c} at $x\pm i0$, $x\in(x_1,x_2)$:
\begin{multline}\label{eq:loop1d}
0=(\W(x+i0)-\W(x-i0))(\W(x+i0)+\W(x-i0)-t^{-1}x)
\\+\omega(\G(x+i0)-\G(x-i0))+\omega^{-1}(\bG(x+i0)-\bG(x-i0))
\end{multline}
Similarly, by setting $y=\omega(1-\omega x)$ in \eqref{eq:loop2c} 
we find:
\begin{gather}
(\W(x+i0)-\W(x-i0))\W(\omega-\omega^{2}x)=\omega^{-1}(\G(x+i0)-\G(x-i0))\label{eq:Gcut}
\\
\intertext{and by exchanging to the roles of $x$ and $y$,}
(\W(x+i0)-\W(x-i0))\W(\omega^{-1}-\omega^{-2}x)=\omega(\bG(x+i0)-\bG(x-i0))\label{eq:bGcut}
\end{gather}
Substituting back into \eqref{eq:loop1d} and
dividing by $\W(x+i0)-\W(x-i0)$, we obtain, for all $x\in(x_1,x_2)$:
\begin{equation}\label{eq:spe}
0=\W(x+i0)+\W(x-i0)-\frac{x}{t}+\omega^{-2}\W(\omega^{-1}-\omega^{-2}x)+\omega^2\W(\omega-\omega^2x)
\end{equation}

Finally we have arrived at an equation considered by Kostov \cite{Kostov-6v}. The relation between our deduction of this equation and Kostov's deduction is described in more detail in Appendix \ref{sec:matrix}. Indeed, this equation had been deduced previously \cite{artic10}, Kostov's major contribution was a solution to this equation. Following Kostov, we consider the function
\begin{align}\begin{split}
U(x):=x\omega \W\left(\frac{1}{\omega+\omega^{-1}}+i\omega x\right)&+x\omega^{-1}\W\left(\frac{1}{\omega+\omega^{-1}}-i\omega^{-1}x\right)\label{eq:Udef}\\
&+\frac{ix^2}{t(\omega^2-\omega^{-2})}  -\frac{x}{t(\omega+\omega^{-1})^2},
\end{split}\end{align}
which is holomorphic in
$\mathbb{C}$ minus the two cuts $( i\omega)^{\pm 1}[x'_1,x'_2]$, where
$x_i'$ is a translate of $x_i$ by an explicit real constant. The reason for considering this function $U(x)$ is that \eqref{eq:spe} takes the following simple form in terms of $U$:
\begin{equation}\label{eq:glue}
U(i\omega(x\pm i0))=U(-i\omega^{-1}(x\mp i0)),
\qquad
x\in (x'_1,x'_2).
\end{equation}
The expansion $\W(x)=x^{-1}+O(x^{-2})$ at $x\to\infty$, yields the following two conditions
\begin{equation}\label{eq:exp}
U(x)=\frac{i}{t(\omega^2-\omega^{-2})} x^2 - \frac{1}{t(\omega+\omega^{-1})^2}x
+O(1/x) \qquad \hbox{as }
x\rightarrow \infty,
\end{equation}
\begin{equation}\label{eq:norm}
\oint_{\mathcal C} \frac{dx}{2\pi x} U(x)=1 ,
\end{equation}
where $\mathcal C$ surrounds the cut $(i\omega)^{-1}[x'_1,x'_2]$ anticlockwise. As we will see, \eqref{eq:glue}, \eqref{eq:exp} and \eqref{eq:norm} contain enough information to determine $U(x)$ exactly.
 
Note that by expanding $U(x)$ at infinity further than \eqref{eq:exp}, i.e., $U(x)=\sum_{i=-2}^\infty U_i x^{-i}$,
we can extract from $U(x)$ the same information as from $\W(x)$. In particular,
\begin{equation}\label{eq:U1}
U_1=
1-(\omega+\omega^{-1})[x^{-2}]\W(x)=1-t(\omega+\omega^{-1})^2\left(1+\Qgf\!\left(t,\omega^{2}+\omega^{-2}\right)\right).
\end{equation}

\subsection{Solution in terms of theta functions}
We now provide a parametric expression for $U(x)$, following \cite{Kostov-6v}.
This expression will involved the classical Jacobi theta function $\theta_1$, which we denote by $\th$:
\begin{align}\label{theta-def}
\th(z)\equiv\th(z,q)&= 2\sin(z) q^{1/8} \prod_{n=1}^{\infty} (1-2\cos(2z)q^n+q^{2n})(1-q^n),\\
=\th(z|\tau)&= -i \sum_{n\in \zs} (-1)^n e^{(n+1/2)^2\pi i\tau+(2n+1)iz},
\end{align}
where $q=e^{2\pi i \tau}$ and $\tau$ has positive imaginary
  part.
  
We will first parametrise the domain $\mathbb{C}\cup\{\infty\}\setminus (i\omega[x'_1,x'_2]\cup -i\omega^{-1}[x'_1,x'_2])$ on which $U(x)$ is meromorphic. To do this, we use the classical result (see for example \cite[Chapter 5, Section 1]{Goluzin1969geometric} for an equivalent statement with ``cylinder" replaced by ``annulus").
\begin{Theorem}
Any doubly connected domain other than the puctured disk and punctured plane is conformally equivalent to some cylinder. 
\end{Theorem}
In our case, this means that there is some $\tau\in i\mathbb{R}^{+}$ and some conformal mapping
\[x:\left\{z\in\mathbb{C}\left|\text{im}(z)\in \left(-\frac{\pi\tau}{2},\frac{\pi\tau}{2}\right)\right.\right\}/\pi\mathbb{Z}\to\mathbb{C}\cup\{\infty\}\setminus (i\omega[x'_1,x'_2]\cup -i\omega^{-1}[x'_1,x'_2]).\]
In particular, this allows us to define $V(z)=U(x(z))$ in this region. By symmetry $x(\overline{z})=\overline{x(z)}$, and we may assume that $x(0)=\infty$ and that $x$ sends the boundary $-\frac{\pi\tau}{2}+\mathbb{R}/\pi\mathbb{Z}$ to $-i\omega^{-1}[x'_1,x'_2]$. Then the relation $x(\overline{z})=\overline{X(z)}$ applied to the boundary implies that for any $r\in\mathbb{R}$, 
\[x\left(r+\frac{\pi\tau}{2}\right)=-\omega^2 x\left(r-\frac{\pi\tau}{2}\right)=e^{-2i\alpha} x\left(r-\frac{\pi\tau}{2}\right),\]
where $\alpha$ is chosen so that $\omega=ie^{-i\alpha}$, which gives $\gamma=\omega^{2}+\omega^{-2}=-2\cos(2\alpha)$. It follows from the identities
  \[
    \th(z+\pi)=-\th(z) \quad \hbox{and} \quad \th(z+\pi\tau)=-e^{-i\pi
      \tau -2i z}\th(z),
  \]
that the expression $x(z)\th(z)/\th(z+\alpha)$
must be an elliptic function of $z$, with $\pi$ and $\pi\tau$ as periods, which has at most one pole (at $z=-\alpha$). In fact this expression must be constant, as non-constant elliptic functions have at least 2 poles. Hence we can define $x$ (analytically extended to all of $\mathbb{C}$) as follows:
  
 Define the mapping $x:\mathbb{C}\to\mathbb{C}$ by
\[
x(z) = x_0 \frac{\th(z+\alpha|\tau)}{\th(z|\tau)}.
\]
The quantities $x_0$ and $\tau$ will
be determined later. Note that $x(z)$ is a meromorphic function whose poles form the lattice $\pi\tau \ZZ+\pi\ZZ$.

Now we determine $V(z)\equiv U(x(z))$. In this context, property~\eqref{eq:glue} is equivalent to the following for $r\in\mathbb{R}$:
\[V\left(r-\frac{\pi\tau}{2}\right)=V\left(r+\frac{\pi\tau}{2}\right).\]
In particular, this implies that $V(z)$ is elliptic with $\pi$ and $\pi\tau$ as periods. Moreover, the only singularity of $V(z)$ is a double pole at $z=0$, due to the double pole of $U(x)$ at $x=\infty$. Together, these properties uniquely define $V(z)$ as a linear transformation of the Weierstrass function: 
%
\[
V(z)\equiv U(x(z))=a+b\wp(z),
\qquad
\wp(z)=\frac{1}{z^2}+\sum_{(m,n)\ne (0,0)} \left(\frac{1}{(z+\pi(m+n\tau))^2}-\frac{1}{\pi^2(m+n\tau)^2}\right).
\]
The parameters $\tau,x_0,a,b$ can be determined by the expansion of $U$ at infinity \eqref{eq:exp} and
the normalization condition~\eqref{eq:norm}.
We note that we can also write $\wp(z)$ in terms of theta functions as follows:
\[\wp(z)=\frac{\th'(z)^2}{\th(z)^2}-\frac{\th''(z)}{\th(z)}+\frac{\th'''(0)}{3\th'(0)}.\]

The three terms of expansion~\eqref{eq:exp} provide three equations which determine $x_{0}$, $b$ and $a$ in terms of $\alpha$ and $\tau$. We ignore the equation coming from the constant term, since this only determines $a$, which plays no role in any further calculations. We are left with the equations:
\begin{equation*}
b=\frac{1}{16t}\frac{\cos(\alpha)}{\sin^{3}(\alpha)}\frac{\th(\alpha)^{2}}{\th'(\alpha)^{2}}
,\qquad
x_0=\frac{\cos\alpha}{2\sin\alpha}\frac{\th'(0)}{\th'(\alpha)}
\end{equation*}
The integral \eqref{eq:norm} can be computed; fixing a mistake in~\cite[App.~B.2]{Kostov-6v} results in
a massive simplification:
\begin{equation}\label{eq:t}
t=
  \frac{\cos\alpha}{64\sin^3\alpha}
\left(
-\frac{\th(\alpha|\tau)\th'''(\alpha|\tau)}{\th'(\alpha|\tau)^2}+\frac{\th''(\alpha|\tau)}{\th'(\alpha|\tau)}
\right).
\end{equation}
The last equation should be understood as an implicit equation for $\tau$ as a function of
$t$; if we want to return to formal power series, 
then it determines $q=e^{2\pi i\tau}$ uniquely once we require $t\sim q$ around $0$,
as claimed in Theorem~\ref{thm:allgamma}.

Finally, by expanding $U(x)$ one order further, one finds
\begin{equation}\label{eq:W1}
t U_1=
\frac{\cos^2\alpha}{96\sin^4\alpha}
\frac{\th(\alpha|\tau)^2}{\th'(\alpha|\tau)^2}
\left(-\frac{\th'''(\alpha|\tau)}{\th'(\alpha|\tau)}
+\frac{\th'''(0|\tau)}{\th'(0|\tau)}\right).
\end{equation}
Theorem \ref{thm:allgamma} then follows due to \eqref{eq:U1}, writing $\Rgf(t,\gamma)=t U_1$, where $\gamma=-2\cos(2\alpha)$.

\section{Proof of the Theorem \ref{thm:allgamma} by guess and check}
\label{sec:proof}
In our derivation of the result, we made an assumption known generally as the ``one cut assumption''. We believe that this could be proven directly as was done for a similar problem in \cite{BE-On}, Lemma 1.1. However, doing so would be very tedious and so we will prove our result in a different way. To be precise, in this section we show that the series $\Wgf$ and $\Hgf$ corresponding to the function $U(x)$ form the unique pair of series satisfying the conditions of Lemma~\ref{lem:uniqueness}, and therefore these are the series in question.

We start by defining the series $W^{(0)}$, $G^{(0)}$ and $\overline{G}^{(0)}$ in terms of $V(z):=a+b\wp(z)$, where $a$ and $b$ are defined as in the previous section. Although $a$ has not been calculated explicitly, we note that its value has no effect on $W^{(0)}(y)$, $G^{(0)}$ and $\overline{G}^{(0)}$ defined below. In the following expressions, we use $c=-\omega-\omega^{-1}$:
\[W^{(0)}(y):=-\frac{1}{2\pi}\int_{0}^{\pi}\frac{V(z-\frac{\pi\tau}{2})x'\left(z-\frac{\pi\tau}{2}\right)}{\left(y+c^{-1}-i\omega x(z-\frac{\pi\tau}{2})\right)x\left(z-\frac{\pi\tau}{2}\right)}dz.\]
\[\overline{G}^{(0)}(y):=-\frac{\omega^{-1}}{2\pi}\int_{0}^{\pi}\frac{V(z-\frac{\pi\tau}{2})x'\left(z-\frac{\pi\tau}{2}\right)}{(y+c^{-1}-i\omega x(z-\frac{\pi\tau}{2}))x\left(z-\frac{\pi\tau}{2}\right)}W^{(0)}\left(-c^{-1}-i\omega^{-1} x\left(z-\frac{\pi\tau}{2}\right)\right)dz,\]
and
\[G^{(0)}(y):=-\frac{\omega}{2\pi}\int_{0}^{\pi}\frac{V(z-\frac{\pi\tau}{2})x'\left(z-\frac{\pi\tau}{2}\right)}{(y+c^{-1}-i\omega x(z-\frac{\pi\tau}{2}))x\left(z-\frac{\pi\tau}{2}\right)}W^{(0)}\left(-c^{-1}-i\omega^3 x\left(z-\frac{\pi\tau}{2}\right)\right)dz.\]
As will become clear in the proofs of Lemmas~\ref{lem:Weq_check} and \ref{lem:cut_eqs_check} these formulae are specifically designed to agree with the definition \eqref{eq:Udef} of $U(x)$ and satisfy equations \eqref{eq:Gcut} and \eqref{eq:bGcut}. We then define
\begin{align*}
\Ggf(x)&:=\frac{1}{x}\G\left(\frac{1}{x}\right)&&\overline{\Ggf}(x):=\frac{1}{x}\bG\left(\frac{1}{x}\right)\\
\Wgf(x)&:=\frac{1}{x}\W\left(\frac{1}{x}\right)&&\Hgf(x,y):=\frac{xy\Wgf(x)\Wgf(y)-\omega x\Ggf(x)-\omega^{-1}y\overline{\Ggf}(y)}{xy-\omega x-\omega^{-1}y}.
\end{align*}
The main result of this section is that this $\Wgf(x)$ and $\Hgf(x,y)$ are the unique series defined in Lemma \ref{lem:uniqueness}. The proof breaks into three parts:
\begin{itemize}
\item $\Wgf(x)$ expands as a series in $\mathbb{C}[x][[t]]$ (Lemma \ref{lem:WGseries}),
\item $\Hgf(x,y)$ expands as a series in $\mathbb{C}[x,y][[t]]$ (Lemma \ref{lem:Hseries_check}),
\item $\Wgf(x)$ and $\Hgf(x,y)$ satisfy \eqref{wloopcombi} and \eqref{hloopcombi} (Lemma \ref{lem:W_H_eqs_check}).
\end{itemize}
We start by proving the equivalent statements for $\W(x)$, $\G(x)$, $\bG(x)$ and $\H(x,y)$.

\begin{Lemma}\label{lem:WGseries}
For $q$ sufficiently small, Each of the three functions $\W$, $\G$ and $\bG$ expands as a series as $y\to\infty$, and this series lies in $y^{-1}\mathbb{C}[y^{-1}][[q]]$. Equivalently, $\Wgf(x)$, $\Ggf(x)$ and $\overline{\Ggf}(x)$ are all series in $\mathbb{C}[x][[q]]\equiv\mathbb{C}[x][[t]]$.
\end{Lemma}
\begin{proof} We start by proving the result for $\W(y)$. In this case, the integrand can be written as
\[
\sum_{n=0}^{\infty}\frac{V(z-\frac{\pi\tau}{2})x'\left(z-\frac{\pi\tau}{2}\right)}{x\left(z-\frac{\pi\tau}{2}\right)} \left(i\omega x\left(z-\frac{\pi\tau}{2}\right)-c^{-1}\right)^{n}y^{-1-n},\]
which converges as long as $|y|>\left|i\omega x\left(z-\frac{\pi\tau}{2}\right)-c^{-1}\right|$. Since $x\left(z-\frac{\pi\tau}{2}\right)$ is finite for $z\in[0,\pi]$, it is uniformly bounded, so the series converges for sufficiently large $y$ independent of $z$. Now, each $y$ coefficient is a series in $q$ with coefficients polynomial in $e^{2iz}$ and $e^{-2iz}$, so the integral simply extracts the constant terms in each of these polynomial and multiplies it by $\pi$.  To see that this series lies in $y^{-1}\mathbb{C}[y^{-1}][[q]]$, it suffices to observe that the expression $i\omega x\left(z-\frac{\pi\tau}{2}\right)-c^{-1}$ expanded as a series in $q$ has no constant term. Indeed, this implies that the coefficient $[y^{-1-n}]\W(y)$ is divisible by $q$, or equivalently that the coefficient $[q^n]\W(y)$ is a polynomial in $y^{-1}$ with degree at most $n+1$.

Now we turn our attention to the functions $\G(y)$. The only extra element to take into account is the term $W^{(0)}\left(-c^{-1}-i\omega^3 x\left(z-\frac{\pi\tau}{2}\right)\right)$ in the integrand. We can directly expand this as a series in $q$ using our series for $\W(y)$. To see that this is a series in $\mathbb{C}[e^{iz},e^{-iz}][[q]]$, it suffices to show that $\left(-c^{-1}-i\omega^3 x\left(z-\frac{\pi\tau}{2}\right)\right)^{-1}$ is in the same ring. Indeed, this follows from the following two easily checked facts:
\begin{itemize}
\item The series $\left(-c^{-1}-i\omega^3 x\left(z-\frac{\pi\tau}{2}\right)\right)$ lies in $\mathbb{C}[e^{2iz},e^{-2iz}][[q]]$.
\item Its leading term as a series in $q$ is the constant $ie^{-i\alpha}$.
\end{itemize}
Since $W^{(0)}\left(-c^{-1}-i\omega^3 x\left(z-\frac{\pi\tau}{2}\right)\right)$ expands as a series in $\mathbb{C}[e^{iz},e^{-iz}][[q]]$, this series converges for sufficiently small $q$, independent of $z$. As in the case of $\W(y)$, it follows that the series for $\G(y)$ around $y\to\infty$ can be extract from the integral by deleting all terms $e^{2niz}$ for $z\neq 0$. This series $\G(y)$ is therefore an element of $y^{-1}\mathbb{C}[y^{-1}][[q]]$. The proof for $\bG(y)$ is essentially identical to the proof for $\G$ since the leading term of $\left(-c^{-1}-i\omega^{-1} x\left(z-\frac{\pi\tau}{2}\right)\right)$ as a series in $q$ is also a constant. 
\end{proof}

\begin{Lemma}\label{lem:one_cut_check}
The three functions $W^{(0)}(y)$, $G^{(0)}(y)$ and $\overline{G}^{(0)}(y)$ are analytic except on a cut $[y_{0},y_{1}]$ on the real axis.
\end{Lemma}
\begin{proof}
First, we will show that for $y\notin\mathbb{R}$, the integrand for $\W(y)$ is holomorphic for $z\in[0,\pi]$, as this will imply that $\W(y)$ is itself holomorphic for $y\notin\mathbb{R}$. The expression 
\[\frac{V(z-\frac{\pi\tau}{2})x'\left(z-\frac{\pi\tau}{2}\right)}{x\left(z-\frac{\pi\tau}{2}\right)}\]
has no poles for $z\in[0,\pi]$, so it suffices to show that $(y+c^{-1}-i\omega x(z-\frac{\pi\tau}{2}))$ has no roots in this range. From the definition of $x(z)$
\[x(z)=x_{0}\frac{\th(z+\alpha)}{\th(z)},\] it follows that for $z\in\mathbb{R}$,
\[\overline{x(z-\pi\tau)}=x(z+\pi\tau)=-\om^{2}x(z+\pi\tau),\]
so $i\omega x(z-\frac{\pi\tau}{2}))\in\mathbb{R}$. Hence $(y+c^{-1}-i\omega x(z-\frac{\pi\tau}{2}))$ can only be $0$ if $y\in\mathbb{R}$.

This shows that the integrand for $\W(y)$ has no poles when $y\in\mathbb{C}\setminus\mathbb{R}$, so $\W(y)$ is holomorphic in this region. it then follows that the integrands for $\G(y)$ and $\bG(y)$ have no poles for $y\in\mathbb{C}\setminus\mathbb{R}$, so $\G(y)$ and $\bG(y)$ are also holomorphic in this region.

Finally, from Lemma~\ref{lem:WGseries}, we know that $\W(y)$, $\G(y)$ and $\bG(y)$ have power series expansions for large $y$. This implies that the cut on the real axis for each of these functions is finite.
\end{proof}

\begin{Lemma}\label{lem:Weq_check}
The function $\W(y)$ satisfies \eqref{eq:Udef}:
\begin{align*}y\omega \W\left(\frac{1}{\omega+\omega^{-1}}+i\omega y\right)&+y\omega^{-1}\W\left(\frac{1}{\omega+\omega^{-1}}-i\omega^{-1}y\right)\\
&=U(y)-\frac{iy^2}{t(\omega^2-\omega^{-2})}  +\frac{y}{t(\omega+\omega^{-1})^2}\end{align*}
\end{Lemma}
\begin{proof}
Using our definition of $\W(y)$, we can rewrite the left hand side as
\begin{align*}&~-\frac{1}{2\pi}\int_{0}^{\pi}\frac{V(z-\frac{\pi\tau}{2})x'\left(z-\frac{\pi\tau}{2}\right)}{x\left(z-\frac{\pi\tau}{2}\right)}\left(\frac{y\omega}{i\omega y-i\omega x(z-\frac{\pi\tau}{2})}+\frac{y\omega^{-1}}{-i\omega^{-1} y-i\omega x(z-\frac{\pi\tau}{2})}\right)\\
&=-\frac{1}{2\pi}\int_{0}^{\pi}\frac{V(z-\frac{\pi\tau}{2})x'\left(z-\frac{\pi\tau}{2}\right)}{x\left(z-\frac{\pi\tau}{2}\right)}\left(\frac{-iy}{y-x(z-\frac{\pi\tau}{2})}+\frac{iy}{y-x(z+\frac{\pi\tau}{2})}\right).\end{align*}
Since
\[\frac{V(z-\frac{\pi\tau}{2})x'\left(z-\frac{\pi\tau}{2}\right)}{x\left(z-\frac{\pi\tau}{2}\right)}\]
is fixed under translating $z$ by $\pi\tau$ and the entire integrand is fixed under $z\to z+\pi$, the integral can be rewritten as an integral over the closed loop $\mathcal{C}$ which travels anticlockwise around the border of the rectangle with corners $-\frac{\pi\tau}{2}$, $\pi-\frac{\pi\tau}{2}$, $\pi+\frac{\pi\tau}{2}$ and $\frac{\pi\tau}{2}$:
\[\frac{1}{2\pi i}\int_{\mathcal{C}}\frac{V(z)x'\left(z\right)}{x\left(z\right)}\frac{y}{x(z)-y}dz.\]
The value of this is precisely the sum of the residues inside the rectangle, where we include the residue at $0$ but not at $\pi$. The poles in this region are at $0$, $-\alpha$ and the unique point $z_{y}$ satisfying $x(z_{y})=y$ (recalling that this is unique as long as $y$ is not on a cut of $U$). The residues at the three poles $0$, $-\alpha$ and $z_{y}$ are $-\frac{\th'(0)y(y\th'(0)-2x_{0}\th'(\alpha))}{x_{0}^2\th(\alpha)^{2}}b$, $-V(\alpha)$ and $V(z_{y})$. Using our equations for $x_{0}$ and $b$, the residue at $0$ can be rewritten as $-i\frac{y^2}{t(\omega^{2}-\omega^{-2})}+\frac{y}{t(\omega+\omega)^2}$. Then adding these residues yields exactly the desired formula.
\end{proof}

\begin{Lemma}\label{lem:cut_eqs_check}
The series $\W(y)$, $\G(y)$ and $\bG(y)$ satisfy equations \eqref{eq:loop1d}, \eqref{eq:Gcut}, \eqref{eq:bGcut} and \eqref{eq:spe} for $y\in\mathbb{R}$.
\end{Lemma}
\begin{proof}
We will prove the four equations in reverse order, starting with \eqref{eq:spe}:
\[\W(y+i0)+\W(y-i0)+\omega^{-2}\W(\omega^{-1}-\omega^{-2}y)+\omega^2\W(\omega-\omega^2y)=\frac{y}{t}.\]
This follows immediately from \eqref{eq:Udef} taking the difference over the cut.  Now we move on to \eqref{eq:bGcut}:
\[(\W(y+i0)-\W(y-i0))\W(\omega^{-1}-\omega^{-2}y)=\omega(\bG(y+i0)-\bG(y-i0)).\]
Equivalently, we need to show that the expresion
\[-\W(y)\W(\omega^{-1}-\omega^{-2}y)+\omega\bG(y)\]
does not have a cut on the real axis. This expression can be written as the following integral:
\begin{multline*}
\frac{1}{2\pi}\int_{0}^{\pi}\frac{V(z-\frac{\pi\tau}{2})x'\left(z-\frac{\pi\tau}{2}\right)}{\left(y+c^{-1}-i\omega x(z-\frac{\pi\tau}{2})\right)x\left(z-\frac{\pi\tau}{2}\right)}\\
\left(\W(\omega^{-1}-\omega^{-2}y)-W^{(0)}\left(-c^{-1}-i\omega^{-1} x\left(z-\frac{\pi\tau}{2}\right)\right)\right)dz.
\end{multline*}
as in the proof of lemma \ref{lem:one_cut_check}, the only possible pole of the integrand occurs when $y+c^{-1}=i\omega x(z-\frac{\pi\tau}{2})$, however there is a root in the expression
\[\W(\omega^{-1}-\omega^{-2}y)-W^{(0)}\left(-c^{-1}-i\omega^{-1} x\left(z-\frac{\pi\tau}{2}\right)\right)\]
at this point, so the integrand does not have a pole. This comletes the proof of \eqref{eq:bGcut}. The proof of \eqref{eq:Gcut} is essentially identical. Finally we will prove that these series satisfy \eqref{eq:loop1d}:
\begin{multline}
0=(\W(y+i0)-\W(y-i0))(\W(y+i0)+\W(y-i0)-t^{-1}y)
\\+\omega(\G(y+i0)-\G(y-i0))+\omega^{-1}(\bG(y+i0)-\bG(y-i0))
\end{multline}
This follows by multiplying both sides of \eqref{eq:spe} by $(\W(y+i0)-\W(y-i0))$ then substituting the formulas \eqref{eq:Gcut} and \eqref{eq:bGcut}.
\end{proof}

\begin{Lemma}\label{lem:Hseries_check}
The series $\Hgf(x,y)$ lies in $\mathbb{C}[x,y][[t]]$.
\end{Lemma}
\begin{proof}
It follows from \eqref{eq:Gcut} and \eqref{eq:bGcut} that the expression
\[\W(y)\W(\omega^{-1}-\omega^{-2}y)-\omega\bG(y)-\omega^{-1}\G(\omega^{-1}-\omega^{-2}y),\]
is holomorphic in $y$.  in $\mathbb{C}$. Moreover, this converges to $0$ as $y\to\infty$, so it must be identically $0$. Converting this to the series $\Wgf$, $\Ggf$ and $\bGgf$ yields
\[\frac{y^2}{\omega^{-1}y-\omega^{-2}}\Wgf(y)\Wgf\left(\frac{y}{\omega^{-1}y-\omega^{-2}}\right)-\omega y\bGgf(y)-\frac{\omega^{-1}y}{\omega^{-1}y-\omega^{-2}}\Ggf\left(\frac{y}{\omega^{-1}y-\omega^{-2}}\right)=0.\]
Hence, the expression
\[xy\Wgf(y)\Wgf\left(x\right)-\omega y\bGgf(y)-\omega^{-1} x\Ggf(x)\]
is $0$ whenever $x$ and $y$ satisfy
\[xy-\omega y-\omega^{-1}x=0.\]
Since the right hand side is a series in $\mathbb{C}[x,y][[t]]$, each $t$ coefficient must be a polynomial in $x$ and $y$ which is sent to $0$ when $xy-\omega y-\omega^{-1}x=0$. This imples that $xy-\omega y-\omega^{-1}x$ is a divisor of each such polynomial, so we can divide by $xy-\omega y-\omega^{-1}x$. Hence \[\Hgf(x,y):=\frac{xy\Wgf(y)\Wgf\left(x\right)-\omega y\bGgf(y)-\omega^{-1} x\Ggf(x)}{xy-\omega y-\omega^{-1}x}\]
lies in $\mathbb{C}[x,y][[t]]$.
\end{proof}

\begin{Lemma}\label{lem:W_H_eqs_check}
The series $\Hgf(x,y)$ and $\Wgf(x)$ satisfy equations \eqref{hloopcombi} and \eqref{wloopcombi}.
\end{Lemma}
\begin{proof}
It follows from \eqref{eq:loop1d} that the expression
\[\W(y)^2-t^{-1}y\W(y)+\omega\G(y)+\omega^{-1}\bG(y)\]
is holomorphic in $\mathbb{C}$. Moreover, this converges to $0$ as $y\to\infty$, so it must be identically $0$. Converting this to the series $\Wgf$, $\Ggf$ and $\bGgf$ yields
\[y^2\Wgf(y)^2-t^{-1}\Wgf(y)+\omega y\Ggf(y)+\omega^{-1}y\bGgf(y)=0.\]
From the definition of $\Hgf(x,y)$, we have $\Hgf(x,0)=\Ggf(x)$ and $\Hgf(0,y)=\bGgf(y)$, so the equation above can be rewritten as
\[\Wgf(y)=t y^2\Wgf(y)^2+t \omega y\Hgf(y,0)+t \omega^{-1}y\Hgf(0,y)=0,\]
which is precisely equation \eqref{wloopcombi}. Moreover, the definition of $\Hgf(x,y)$ can be rewritten as
\[(xy-\omega y-\omega^{-1}x)\Hgf(x,y)=xy\Wgf(y)\Wgf\left(x\right)-\omega y\Hgf(0,y)-\omega^{-1} x\Hgf(x,0),\]
which is precisely equation \eqref{hloopcombi}.
\end{proof}

Therefore, $\Hgf(x,y)$ and $\Wgf(x)$ are the unique series described by Lemma~\ref{lem:uniqueness}, which enumerate the classes partial of orientations described in Section \ref{sec:tutte}. This completes the proof our results in section \ref{sec:derivation} and, as a consequence, we have proved Theorem \ref{thm:allgamma}:
\begin{align*}
t&=  \frac{\cos\alpha}{64\sin^3\alpha}
\left(
-\frac{\th(\alpha,q)\th'''(\alpha,q)}{\th'(\alpha,q)^2}+\frac{\th''(\alpha,q)}{\th'(\alpha,q)}
\right),\\
\Rgf(t,\gamma)&=\frac{\cos^2\alpha}{96\sin^4\alpha}
\frac{\th(\alpha,q)^2}{\th'(\alpha,q)^2}
\left(-\frac{\th'''(\alpha,q)}{\th'(\alpha,q)}
+\frac{\th'''(0,q)}{\th'(0,q)}\right),\\
\Qgf(t,\gamma)&= \frac{1}{(\gamma+2)t^2}\left( t-(\gamma+2)t^2-\Rgf(t,\gamma)\right).
\end{align*}

\section{Differential equation}
\label{sec:DE}
This section is dedicated to the derivation of a differential equation
satisfied by the functional inverse of our generating series $\Rgf\equiv\Rgf(t,\gamma)$, which we denote by $t(\Rgf)$. This equation will be useful in the next section, where we derive modular properties of our solution.

Since both $\Rgf$ and $t$ have constant term $0$ and linear term $q$ as series in $q$, any of the $t$, $q$ and $\Rgf$ can be written as a series in any of the others. The initial terms of these six series are given below
\begin{align*}
t(q)&=q-6(\omega^2+1+\omega^{-2})q^2+3(9\omega^4+20\omega^2+26+20\omega^{-2}+9\omega^{-4})q^3+O(q^4),\\
t(\Rgf)&=\Rgf+(\omega+\omega^{-1})^2\Rgf^2+2(\omega+\omega^{-1})^2(2\omega^{2}+3+\omega^{-2})\Rgf^3+O(\Rgf^4),\\
q(t)&=t+6(\omega^2+1+\omega^{-2})t^2+3(15\omega^4+28\omega^2+46+28\omega^{-2}+15\omega^{-4})t^3+O(t^4),\\
q(\Rgf)&=\Rgf+(7\omega^2+8+7\omega^{-2})\Rgf^2+(61\omega^4+134\omega^2+206+134\omega^{-2}+64\omega^{-4})\Rgf^3+O(t^4),\\
\Rgf(q)&=q-(7\omega^2+8+7\omega^{-2})q^2+(37\omega^4+90\omega^2+118+90\omega^{-2}+37\omega^{-4})q^3+O(q^4),\\
\Rgf(t)&=t-(\omega+\omega^{-1})^2 t^2-2(\omega+\omega^{-1})^2(\omega+1+\omega^{-1})(\omega-1+\omega^{-1})t^3+O(t^4).
\end{align*}

We also introduce an auxiliary series $A$:
\begin{equation}\label{eq:A}
A=\frac{\th'(\alpha|\tau)}{\th(\alpha|\tau)}
\end{equation}
where the prime (as always) indicates differentiation with respect to the first parameter.
On the other hand, we denote by $D$ the differential operator
\[
D := q \frac{d}{dq}=\frac{q(t)}{q'(t)}\frac{d}{dt}
=
\frac{1}{2\pi i} \frac{d}{d\tau}
\]
where $q=e^{2\pi i\tau}$.

We begin with the following lemma which allows both $\Rgf(t)$ and $t$ to be written in terms of $A$:
\begin{Lemma}\label{lem:id}
The following formulae hold:
\begin{align}\label{eq:id1}
t&=-\frac{\cos\alpha}{8\sin^3\alpha}\,D(A^{-1})
\\
D\Rgf&=-\frac{\cos^2\alpha}{8\sin^4\alpha}\, D^2(A^{-1}) A^{-1}\label{eq:id2}
\end{align}
\end{Lemma}
\begin{proof}
The first identity easily follows from \eqref{eq:t} and \eqref{eq:A}
by using the heat equation satisfied by $\th$:
\[
\th''(z|\tau)+8 D\th(z|\tau)=0
\]

As to the second identity, we again use the heat equation to convert $\tau$-derivatives of $\th(z,\tau)$ to $z$-derivatives. Writing the resulting expression in terms of $A$ and its derivatives, we find
\[D\Rgf+\frac{\cos^2\alpha}{8\sin^4\alpha}\, D^2(A^{-1}) A^{-1}=\frac{\cos^2\alpha}{1536\sin^4\alpha}(2A^{-2}f(\alpha)+A^{-3}f'(\alpha)),\]
where $f(z)$ is given by
\[f(z)=-\frac{4 \th'''(0|\tau) }{\th'(0|\tau)}A'(z)+\frac{\th^{(5)}(0|\tau)}{\th'(0|\tau)}-\frac{\th'''(0|\tau)^2}{\th'(0|\tau)^2}+A'''(z)+6 A'(z)^2,\]
so it suffices to show that $f(z)=0$. To show this, we first note that the periods $\pi$ and $\pi\tau$  of $A'(z)$ are also periods of $f(z)$. Moreover, the only possible pole of $f(z)$ is the pole $0$ of $A(z)$, however, expanding $f(z)$ directly as a series around $z=0$, we find that $f(z)=O(z)$. It follows that $f(z)$ is an elliptic function with no poles, so it must be constant. Moreover, since $f(0)=0$, it follows that $f$ is the zero function.
\end{proof}

We can now state our result:
\begin{Proposition}
$t(\Rgf)$ satisfies the differential equation
\begin{equation}\label{eq:diffeq}
\frac{d^2 t}{d\Rgf^2} - \Sgf\, t =0
\end{equation}
where
\begin{equation}\label{eq:defS}
\Sgf=\frac{8\sin^4\alpha}{\cos^2\alpha}\frac{A^2}{D\Rgf}
\end{equation}
\end{Proposition}
\begin{proof}
Write
\[
\frac{dt}{d\Rgf}=\frac{Dt}{D\Rgf}=\tan\alpha\,A
\]
where the last equality follows from the application of both formulae \eqref{eq:id1}--\eqref{eq:id2}.
Differentiate once more:
\[
\frac{d^2t}{d\Rgf^2}=\tan\alpha\,\frac{DA}{D\Rgf}=-\tan\alpha\, A^2\frac{D(A^{-1})}{D\Rgf}=\Sgf\,t
\]
where we applied again \eqref{eq:id1}.
\end{proof}

For the sake of completeness, we mention a similar equation satisfied by $A$:
\begin{Proposition}
$A(\Rgf)$ satisfies the differential equation
\begin{equation}\label{eq:diffeq2}
\frac{d^2 A}{d\Rgf^2} - \Tgf \frac{dA}{d\Rgf} - \Sgf\, A =0
\end{equation}
where $\Tgf=\frac{d\Sgf/d\Rgf}{\Sgf}$.
Furthermore, any solution of this equation is
a linear combination of $A(\Rgf)$ and $\tau(\Rgf) A(\Rgf)$.
\end{Proposition}
\begin{proof}
The proof of \eqref{eq:diffeq2} is elementary (and similar to that
of \eqref{eq:diffeq}), so that we shall skip it.
Consider now a general solution of \eqref{eq:diffeq} of the form
$f A$, and write the differential equation satisfied by
the Wronskian
\[
{\sf Wr}:=\left|
\begin{matrix}
A&fA\\
\frac{dA}{d\Rgf}&\frac{d(fA)}{d\Rgf}
\end{matrix}
\right|
=A^2 \frac{df}{d\Rgf}
\]
From \eqref{eq:diffeq} it satisfies $\frac{d \log {\sf Wr}}{d\Rgf}=\Tgf=\frac{d\log \Sgf}{d\Rgf}$, which means
$\sf Wr$ is proportional to $\Sgf$. Substituting \eqref{eq:defS}, we conclude
that $\frac{df}{d\Rgf}$ is proportional to $\frac{d\tau}{d\Rgf}$, so that
$f=a \tau+b$ for some constants $a,b$. 
\end{proof}

\section{Modular properties of the solution}\label{sec:modular_properties}
In this section we show that both $\Rgf(\tau)$ and $\Sgf(\tau)$ are {\em modular functions} in the cases when $\alpha$ is a rational multiple of $\pi$, implying that in these cases there is some non-zero polynomial $P$ satisfying $P(\Sgf,\Rgf)=0$. As a consequence, we show the following theorem:
\begin{Theorem} If $\alpha=\frac{M}{N}\pi$ for some $M,N\in\mathbb{Z}$ (with $2M/N\notin\mathbb{Z}$), then the series $t(\Rgf)$ is D-finite, meaning that it satisfies a non-trivial differential equation with coefficients polynomial in $\Rgf$. 
\end{Theorem}
In the following section we determine this differential equation explicitly for certain specific values of $\alpha$. We will assume in this and the next section that $\alpha=\frac{M}{N}\pi$.

If $\alpha\in\pi\mathbb{Z}$ Then the function $A$ is not defined, as the denominator $\th(\alpha|\tau)$ in its definition is $0$. If $\alpha\in\frac{\pi}{2}+\pi\mathbb{Z}$, then $\Sgf$ is not well defined. For the remainder of this section we will assume that $2\alpha\notin\pi\mathbb{Z}$, we note then that all series are well defined and are not identically $0$.


Consider the standard action of the modular group $\PSL_{2}(\mathbb{Z})$ on the upper half plane $\mathbb{H}$, defined by
\[\begin{bmatrix}a&b\\c&d\end{bmatrix}\cdot\tau=\frac{a\tau+b}{c\tau+d}.\]
And consider the subgroup $\Gamma_{1}(N)$ of $\PSL_{2}(\mathbb{Z})$ of these matrices for which $(c,d)\equiv (0,1)$ modulo $N$. We will prove in this section that $\Rgf(\tau)$ and $\Sgf(\tau)$ are modular functions in the sense that they are fixed under the action of $\Gamma_{1}(N)$.

We start with a result of Hermite \cite{Hermite1858}, which states that for any matrix
\[\gamma=\begin{bmatrix}a&b\\c&d\end{bmatrix}\in \SL_{2}(\mathbb{Z}),\]
the following equation holds:
\begin{equation}\label{thmodular}(c\tau+d)^{-\frac{1}{2}}\th\left(z,\frac{a\tau+b}{c\tau+d}\right)=e^{K\pi i/4}\exp\left(\frac{ic(c\tau+d)}{\pi}z^2\right)\th(z(c\tau+d),\tau).\end{equation}
for some $K\in\mathbb{Z}$ depending only on $\gamma$. 
If $\gamma\in\Gamma_{1}(N)$, we can write $c=N c_{1}$ and $d=Nd_{1}+1$, and we can relate $\th(z(c\tau+d),\tau)$ to $\th(z,\tau)$. this leads to the equation
\[(c\tau+d)^{-\frac{1}{2}}\th\left(\alpha,\frac{a\tau+b}{c\tau+d}\right)=(-1)^{Md_{1}}e^{K\pi i/4}\exp\left(-\frac{M^{2}}{N}c_{1}i-M^{2}c_{1}d_{1}i\right)\th(\alpha,\tau)\]
in these cases. 
For our puposes it will be sufficient to notice that the factor on the right hand side does not depend on $\tau$. Calling this factor $J$, the equation becomes
\begin{equation}\label{thalphamodular}(c\tau+d)^{-\frac{1}{2}}\th\left(\alpha,\frac{a\tau+b}{c\tau+d}\right)=J\th(\alpha,\tau).\end{equation}
Taking the log derivative of \eqref{thmodular}, we see that
\[\frac{\th'}{\th}\left(z,\frac{a\tau+b}{c\tau+d}\right)=\frac{2izc(c\tau+d)}{\pi}+(c\tau+d)\frac{\th'}{\th}(z(c\tau+d),\tau).\]
Again, if $\gamma\in\Gamma_{1}(N)$, we can relate $\th\left(\alpha(c\tau+d),\tau\right)$ to $\th(\alpha,\tau)$, which yields
\[\frac{\th'}{\th}\left(\alpha,\frac{a\tau+b}{c\tau+d}\right)=(c\tau+d)\frac{\th'}{\th}(\alpha,\tau),\]
so $A(\tau)=\frac{\th'(\alpha,\tau)}{\th(\alpha,\tau)}$ is a modular function of weight 1 on $\Gamma_1(N)$. Multiplying this by \eqref{thalphamodular} yields
\[\th'\left(\alpha,\frac{a\tau+b}{c\tau+d}\right)=J(c\tau+d)^{3/2}\th'(\alpha,\tau).\]
We then take the log derivative with respect to $\tau$ to obtain
\[(c\tau+d)^{-2}\frac{\th'''}{\th'}\left(\alpha,\frac{a\tau+b}{c\tau+d}\right)=\frac{3}{2}\frac{c}{c\tau+d}+\frac{\th'''}{\th'}(\alpha,\tau).\]
We can prove that the same equation holds when we replace $\alpha$ by $0$ on the whole group $\SL_{2}(\mathbb{Z})$. It follows that $\Rgf(\tau)$ is a modular function of weight $0$ on $\Gamma_1(N)$, so it is algebraically related to the $j$-invariant.

As an immediate consequence, the function $\Sgf(\tau)$, given by
\[\Sgf=\frac{8\sin^4\alpha}{\cos^2\alpha}\frac{A^2}{D\Rgf}\]
is also a modular function on $\Gamma_1(N)$, so it algebraically related to the $j$-invariant. In particular, this implies that $\Sgf$ is an algebraic function of $\Rgf$, so \eqref{eq:diffeq} is a differential equation for $t(\Rgf)$ with coefficients algebraic in $\Rgf$.

Note that $\Rgf(\tau)$ and $\Sgf(\tau)$ being modular functions of this form means that they can be seen as meromorphic functions on the space $\mathbb{H}/\Gamma_1(N)$. This space naturally identifies
with the moduli space of elliptic curves with a marked point of order $N$ \cite[\S 1.5]{DS-modular}.

It is known that for $N\leq10$ and $N=12$, the moduli space $\mathbb{H}/\Gamma_1(N)$
is of genus $0$ (cf~\cite[p112]{FK-elliptic}, where $\Gamma_1(N)$ is denoted $G(N)$). This implies that there is a modular function $h(\tau)$ (so-called Hauptmodul) such that all modular functions are rational functions of $h$
(see e.g.~\cite{Yang-haupt} for some explicit formulae). Hence in these cases, $t$ is D-finite of order $2$ as a series in $h$, while $\Rgf$ is a rational function of $h$. More generally, we can use the following classical result (see \cite{Poincare1884,Schmidt75,JPP19}) to show that $t$ is D-finite in $\Rgf$, though the order of the differential equation will, in general, be greater than 2.

\begin{Proposition}\label{prop:classical_Dfinite}
If a power series $f(r)$ satisfies a non-trivial linear differential equation in $r$, with coefficients {\em algebraic} in $r$, then it satisfies a linear differential equation with coefficients {\em polynomial} in $r$ - in other words, $f(r)$ is a D-finite power series.
\end{Proposition}
For the convenience of the reader we outline a proof of this theorem below.
\begin{proof}
By the assumption, there is some equation
\[f^{(n)}(r)+\sum_{j=0}^{n-1}a_{j}f^{(j)}(r),\] where each $a_{j}$ is algebraic in $r$. Equivalently,
\[f^{(n)}(r)\subset\mathbb{R}(r,a_{0},\ldots,a_{n-1})\langle f,f',\ldots,f^{(n-1)}\rangle.\]
Taking successive derivatives, we see that $f^{(k)}(r)$ is in this set for all positive integers $k$. The next ingredient is that, because $a_{j}$ is algebraic, $\mathbb{R}(r,a_{j})=\mathbb{R}(r)[a_{j}]$, and this is a finite dimensional vector space over $\mathbb{R}(r)$. In the end, this means that all of the derivatives $f^{(k)}(r)$ lie in a finite dimensional vector space, so they are linearly dependant.\end{proof}

\section{Particular cases}
\label{sec:particular_cases}
The previous section showed that when $\alpha/\pi$ is a rational number $M/N$,
the generating series possess remarkable additional properties. We now
investigate some special cases corresponding to some values of the denominator
$N$. In particular, we recover some remarkable identities of Ramanujan
\cite{Rama-modular,Berndt-Rama}.


\subsection{\texorpdfstring{$\gamma=1$}{gamma=1}}
This case is natural from a combinatorial perspective as it corresponds to pure enumeration of quartic Eulerian orientations \cite{BoBoDoPe,elvey-guttmann17,BM_AEP},
and to the value $\alpha=\pi/3$, i.e., $N=3$.
\newcommand\oeis[1]{\href{https://oeis.org/#1}{#1}}

Define $[k]=\eta(q^k)$ where $\eta(q)=q^{1/24}\prod_{n=1}^\infty (1-q^n)$ and define $h$ by:
\[h=\left(\frac{[3]}{[1]}\right)^{12}.\tag{\oeis{A121590}}\]
This is a standard choice of Hauptmodul for $\Gamma_1(3)$,\footnote{One may consult~\cite[Tables~2 and 3]{Maier-haupt}
which provides Hauptmoduls for $\Gamma_0(N)$, noting that $\Gamma_1(N)\cong \Gamma_0(N)$ for $N=2,3,4,6$.}
i.e., a particular rational parameterization of the moduli
space $\mathbb H/\Gamma_1(3)$.
One easily computes the various $q$-series using the product form of $\th(z)\equiv\th(z,\tau)$:
\begin{align*}
\th(\alpha)/\sqrt{3}&=[3]
\\
\sqrt{3}A &= 
1+6\sum_{n=1}^\infty\frac{q^{n}}{1+q^{n}+q^{2n}}
 = \sum_{m,n\in\ZZ} q^{m^2+mn+n^2}
               \tag{\oeis{A004016}; see also \oeis{A002324}}
\\
\Sgf/6&=\frac{(1+27h)^2}{h}=\frac{1}{\Rgf(1-27\Rgf)}
\tag{\oeis{A030197}}
\\
\Rgf&=\frac{h}{1+27h}=\frac{1}{27}\left(\frac{\sum_{m,n\in\ZZ+1/3} q^{m^2+mn+n^2}}{\sum_{m,n\in\ZZ} q^{m^2+mn+n^2}}\right)^3
\tag{numerator is \oeis{A106402}}
\end{align*}
(see e.g.~\cite{BBG-rama} for the identities relating these various $q$-series)
where we have also provided the OEIS references
for convenience. Note that $\Rgf$ and $h$ are rational functions of each other, so $\Rgf$ is a Hauptmodul. 
As a consequence, the differential equation \eqref{eq:diffeq} has rational coefficients;
explicitly,
\begin{equation}
\frac{d^2t}{d\Rgf^2}-\frac{6}{\Rgf(1-27\Rgf)}\, t=0.
\end{equation}
This is precisely the $Q$-form of a hypergometric differential equation.
One finds that the solution with correct initial conditions is
\[
t=\Rgf\ {}_2F_1(1/3,2/3;2;27\Rgf),
\]
which is exactly the expression used to define $\Rgf(t,1)$ in \cite{BM_AEP}.
Furthermore, we note that the
 differential equation \eqref{eq:diffeq2}
is also of hypergeometric type:
\[
\frac{d^2A}{d\Rgf^2}-\frac{1}{27-\Rgf}\frac{dA}{d\Rgf}-\frac{1}{\Rgf(1-27\Rgf)} A=0
\]
and its solution is given by
$\sqrt{3}A={}_2F_1(1/3,2/3;1;27\Rgf)$. Its other independent solution
$\tau A$ is given by $\tau A={}_2F_1(1/3,2/3;1;27(1-\Rgf))$. This implies Ramanujan's identity \cite{Rama-modular}
\[
q=\exp\left(-\frac{2\pi}{\sqrt{3}} \frac{{}_2F_1(1/3,2/3;1|27(1-\Rgf))}{{}_2F_1(1/3,2/3;1|27\Rgf)}\right).
\]

\subsection{\texorpdfstring{$\gamma=0$}{gamma=0}}
Combinatorially this corresponds to the enumeration of quartic Eulerian orientations with no alternating vertices. This case was considered (and solved) in \cite{BM_AEP}, as they proved bijectively that it also corresponds to the enumeration of Eulerian oriantations by edges (with no restriction on vertex degrees). This problem is also equivalent to a special case of the ABAB model \cite{artic05}, as can be seen by comparing (1.4) in \cite{artic05} at $c=d=0$ to \eqref{eq:Strat} (at $\gamma=0$).

In this case, $\alpha=\pi/4$, so $N=4$.

The expressions are very similar to the case $N=3$ and match with
another famous identity of Ramanujan. We have:
\begin{align*}
\th(\alpha)/\sqrt{2}&=q^{1/8} \frac{[1][4]}{[2]}\tag{\oeis{A106459}}
\\
A &= 
1+4\sum_{m=1}^\infty\frac{q^n}{1+q^{2n}}
=\Big(\sum_{n\in\ZZ}q^{n^2}\Big)^2
=\frac{[2]^{10}}{[1]^4[4]^4}
\tag{\oeis{A004018}; see also
\oeis{A002654}}
\\
h&=\left(\frac{[4]}{[1]}\right)^{8}
\tag{\oeis{A092877}}
\\
\Sgf/4&=\frac{(1+16h)^2}{h}=\frac{1}{\Rgf(1-16\Rgf)}=\left(\frac{[2]^2}{[1][4]}\right)^{24}
\tag{\oeis{A097340}}
\\
\Rgf&=\frac{h}{1+16h}=\frac{[1]^8[4]^{16}}{[2]^{24}}=\frac{1}{16}\left(\frac{\sum_{n\in\ZZ+1/2} q^{n^2}}{\sum_{n\in\ZZ} q^{n^2}}\right)^4
\tag{\oeis{A005798}}
\end{align*}
Once again, $\Rgf$, just like $h$, provides a rational parameterization
of $\mathbb H/\Gamma_1(4)$,
and the differential equation \eqref{eq:diffeq}
\[
\frac{d^2t}{d\Rgf^2}-\frac{4}{\Rgf(1-16\Rgf)}\, t=0
\]
is the $Q$-form of a hypergometric differential equation, with solution
\[
t=\Rgf\ {}_2F_1(1/2,1/2;2;16\Rgf).
\]
This is the expression used to define $\Rgf(t,0)$ in \cite{BM_AEP}.
Furthermore, we note that the differential equation \eqref{eq:diffeq2}
is also of hypergeometric type,
and its two independent solutions are given by
$A={}_2F_1(1/2,1/2;1;16\Rgf)$ and
$\tau A={}_2F_1(1/2,1/2;1;16(1-\Rgf))$; leading to Ramanujan's identity \cite{Rama-modular}
\[
q=\exp\left(-2\pi \frac{{}_2F_1(1/2,1/2;1|16(1-\Rgf))}{{}_2F_1(1/2,1/;1|16\Rgf)}\right).
\]

For $N\ge 5$, one can check by direct inspection that $\Rgf$ is {\em not}\/ a Hauptmodul for $\mathbb H/\Gamma_1(N)$.
In what follows we provide examples
of explicit differential equations for $N=5,6$. Let us first discuss $N=6$, since it is simpler. 

\subsection{\texorpdfstring{$\gamma=-1$}{gamma=-1}}
For $\gamma=-1$, i.e., $\alpha=\pi/6$, the coefficients of $\Qgf(t,-1)$ are given by the number of quartic Eulerian orientations with an even number of alternating vertices minus the number with an odd number of alternating vertices. 

The compactified moduli space $\overline{\mathbb H/\Gamma_1(5)}\cong \mathbb P^1$ 
has $4$ logarithmic cusps, corresponding to $\tau=i\infty,0,1/2,1/3$.
One can check by using modular transformations of theta functions that $\Rgf$ and $\Sgf^{-1}$ are regular at the first three, whereas they have
poles at the last:
\begin{align*}
\Rgf &= \frac{1}{32 v^2}+\frac{1}{16 v}+\frac{9}{32}+\frac{5v}{8}+\frac{19v^2}{16}+\frac{9v^3}{16}
-\frac{15v^4}{8}-5v^5+O(v^6)
\\
\Sgf^{-1} &= -\frac{1}{256 v^4}-\frac{3}{256 v^3}-\frac{3}{256 v^2}+\frac{13}{256 v}+\frac{9}{32}+\frac{129
   v}{256}-\frac{37 v^2}{128}+O(v^3)
\\
\text{where }\tau&=\frac{1}{3}+i\epsilon,
\quad v=e^{\frac{2\pi i}{3}-\frac{\pi}{9\epsilon}}.
\end{align*}
As a consequence, the functions $1,\Rgf,\Rgf^{2},\Rgf^{3},\Rgf^{4},\Sgf^{-1},\Sgf^{-1}\Rgf,\Sgf^{-1}\Rgf^2,\Sgf^{-2}$ each have a pole of order at most $8$ at this cusp. 
Hence there is some non-trivial linear combination of these functions wich no such pole. By solving the relevant linear equations, we find the following example:
\[
Q(\Rgf,\Sgf^{-1})
=
256 \Rgf^4-264 \Rgf^3+128 \Rgf^2 \Sgf^{-1}+3 \Rgf^2-64 \Rgf \Sgf^{-1}+5
   \Rgf+16 \Sgf^{-2}-10 \Sgf^{-1}.
\]
$Q(\Rgf,\Sgf^{-1})$ is holomorphic everywhere
on $\overline{\mathbb H/\Gamma_1(5)}$, so it must be a constant function. Moreover it vanishes at $\tau=i\infty$ ($q=0$), so it is identically zero.

In other words, we have found a polynomial equation relating $\Rgf$ and $\Sgf$, namely
\[
256 \Rgf^4 \Sgf^2-264 \Rgf^3 \Sgf^2+3 \Rgf^2 \Sgf^2+128 \Rgf^2 \Sgf+5 \Rgf \Sgf^2-64 \Rgf \Sgf-10 \Sgf+16=0.
\]
This is the equation of an elliptic curve with a nodal singularity, so that it is of
(geometric) genus $0$, as expected. In this case the rational parameterization is easily
found by hand:
\begin{align*}
\Rgf&=h(1+2h),
\\
\Sgf/2&=\frac{1}{h(1+h)(1+4h)(1-8h)}.
\end{align*}
and agrees with a known Hauptmodul for $\Gamma_1(6)\cong \Gamma_0(6)$:
\[
h=\left(\frac{[1][6]^3}{[2][3]^3}\right)^3.
\tag{\oeis{A123633}}
\]
as can be proven in a similar way as above, by checking the relations between $h$ and $\Rgf$, $\Sgf$ at all would-be poles.

Replacing, we find the differential equation
\[
\frac{d^2t}{dh^2}-\frac{4}{1+4h}\frac{dt}{dh}-\frac{2(1+4h)}{h(1+h)(1-8h)}t=0.
\]
This equation has 5 singularities, so is not of hypergeometric type.

\subsection{\texorpdfstring{$\gamma=\frac{1+\sqrt{5}}{2}$}{gamma=(1+sqrt5)/2}}
$N=5$, $M=1,2$ correspond respectively to
$\gamma=\frac{1\mp\sqrt{5}}{2}$, which are trivially related to each other;
we pick $\gamma=\frac{1+\sqrt{5}}{2}$. The reasoning is similar to the previous section, and we skip the details.

Again, one can check that $\Rgf$ and $\Sgf$ satisfy a polynomial equation. We will not need the exact form, but its Newton polygon is
\begin{center}
\tikzset{vertex/.style={circle,draw=none,fill=red,inner sep=2pt,fill opacity=1}}
\begin{tikzpicture}
\draw[->,gray] (0,0) -- (0,3.5) node[above,black] {$\Sgf$};
\draw[->,gray] (0,0) -- (6.5,0) node[right,black] {$\Rgf$};
\foreach \x in {0,...,6}
\foreach \y in {0,...,3}
\node[circle,fill=gray,inner sep=1pt] at (\x,\y) {};
\draw[very thick,black,fill=blue,fill opacity=0.25] (0,0) -- (6,3) -- (2,3) -- (0,1) -- cycle;
\foreach \x/\y in {0/0,0/1,1/1,2/1,1/2,2/2,3/2,4/2,2/3,3/3,4/3,5/3,6/3} \node[vertex] at (\x,\y) {};
\end{tikzpicture}
\end{center}
Each red dot at a point $(j,k)$ represents a monomial $c\Rgf^{j}\Sgf^{k}$ in the polynomial.
The resulting planar curve of values $(\Rgf,\Sgf)$ has three singular points, so that its (geometric) genus is $0$, as expected.
Rather than working out its rational parameterization, one can use the known expression for a Hauptmodul
of $\Gamma_1(5)$, namely
\[
h=q\prod_{n=0}^\infty \frac{(1-q^{5n+1})(1-q^{5n+4})}{(1-q^{5n+2})(1-q^{5n+3})}
\tag{\oeis{A078905}},
\]
where, as in previous sections, we have chosen $h=q+O(q^2)$.
In terms of this Hauptmodul $h$, one has
\begin{align*}
\Rgf&=\frac{h\left(1-\frac{1+\sqrt{5}}{2}h\right)}{\left(1+(2+\sqrt{5})h\right)^3}
\\
\Sgf&=(5+\sqrt{5})\frac{\left(1+(2+\sqrt{5})h\right)^6}{h\left(1-\frac{11-5\sqrt{5}}{2}h\right)
\left(1-\frac{11+5\sqrt{5}}{2}h\right)^2\left(1-\frac{\sqrt{5}-1}{2}h\right)}
\end{align*}
The differential equation is then
\begin{multline*}
\frac{d^2 t}{dh^2}+\frac{13+7\sqrt{5}-(36\sqrt{5}+82)h+(29+13\sqrt{5})h^2}%
{\left(1-\frac{\sqrt{5}-1}{2}h\right)\left(1+(2+\sqrt{5})h\right)\left(1-\frac{11+\sqrt{5}}{2}h\right)}
\,\frac{dt}{dh}
\\
-10(3+\sqrt{5})\frac{\left(1+(\sqrt{5}+2)h\right)^4}{\left(1-\frac{11+5\sqrt{5}}{2}h\right)^2
\left(1-\frac{11-5\sqrt{5}}{2}h\right)^2}
\,t=0
\end{multline*}
which has 5 singularities, 
so is once again not hypergeometric.

\bigskip
\goodbreak

\appendix

\section{The matrix model approach}\label{sec:matrix}
In this appendix we describe the matrix model approach and how one can use it to deduce the equations \eqref{eq:loop1c} and \eqref{eq:loop2c}.
\subsection{Definition of the model}
Following \cite{artic10} and \cite{Kostov-6v}, we introduce the following
matrix integral:
\begin{equation}\label{eq:defZ}
Z_N=\int \d X \d X^\dagger \exp\left[N\tr\left(
-XX^\dagger + t X^2X^\dagger{}^2+{\gamma t\over 2}(XX^\dagger)^2\right)\right]
\end{equation}
where integration is over $N\times N$ complex matrices,
and $X^\dagger$ denotes the conjugate transpose of $X$. Then
\begin{equation}\label{eq:F}
2t\frac{\partial}{\partial t}\log Z_N = \sum_{g\ge 0} \Qgf^{(g)}(t,\gamma) N^{2-2g},
\end{equation}
where each series $\Qgf^{(g)}(t,\gamma)$ is the genus $g$ analogue of $\Qgf(t,\gamma)\equiv\Qgf^{(0)}(t,\gamma)$.

The first step is to rewrite $Z_N$:
\begin{equation}\label{eq:Strat}
Z_N=\int \d A \d X\d X^\dagger \exp\left[N\tr\left(
-{1\over 2t}A^2-XX^\dagger
+ (\omega XX^\dagger +
\omega^{-1}X^\dagger X)A
\right)\right]
\end{equation}
where $A$ is integrated over Hermitian matrices with its flat measure and
the same normalization as before (in particular $Z_N(t=0):=1$).
Performing the Gaussian integral over $A$, we recover \eqref{eq:defZ}.

Combinatorially, this corresponds to the bijection depicted in Figure \ref{fig:two vertex types transformations}, where each
4-valent vertex is split into a pair of three valent vertices such that each resulting vertex is incident to one incoming edge, one outgoing edge and one undirected edge. In the resulting map there are two types of vertices, shown in Figure \ref{fig:two_A_vertex_types}, which we call {\em right turn vertices} and {\em left turn vertices}. Each right turn vertex is given weight $\omega^{-1}$, each left turn vertex is given weight $\omega$ and each undirected edge is given weight $t$. The undirected edges then correspond to the matrix $A$ in \ref{eq:Strat}.


For any function $F$ of $X$, $X^\dagger$, $A$, denote
\[
\left<F\right> := Z_N^{-1} \int \d A\d X\d X^\dagger \,F\, e^{NS},
\qquad
S:=\tr\left(
-{1\over 2t}A^2
-XX^\dagger
+
(\omega XX^\dagger +
\omega^{-1}X^\dagger X)A
\right)
\]

We are then interested in the quantity:
\begin{equation}\label{eq:defW}
W(x)=\left< \tr \frac{1}{x-A}\right>
\end{equation}
which has the following interpretation. Write 
\[
W(x)=\sum_{g\ge 0} W^{(g)}(x) N^{1-2g},
\qquad
W^{(g)}(x)=\frac{\delta_{g,0}}{x}+\sum_{n\ge 1} \frac{W^{(g)}_n}{x^{n+1}}
\]
Then $W_n^{(g)}$ is the generating series of {\em rooted} partial Eulerian orientations of genus $g$ in which each non-root vertex is either a right turn vertex or a left turn vertex and the root vertex is incident to exactly $n$ undirected edges and no directed edges. 

\subsection{Loop equations}
In this section we use the matrix model to rederive the equations \eqref{eq:loop1c} and \eqref{eq:loop2c}, which characterise the series $\W(x)$. Kostov derived equation \eqref{eq:W1} in a slightly different way, by seeing the function $\rho(x)=\W(x+0i)-\W(x-0i)$ as being related to the $N\to\infty$ spectral density of the eigenvalues of the matrices defining $Z_{N}$. Kostov's perspective is perhaps harder to make rigorous, though it does give an intuition behind the one cut assumption.

Following a standard approach 
(see e.g.~\cite{BE-On} for a similar model),
we write {\em loop equations}\/ satisfied by $W(x)$.

First we introduce some auxiliary correlation functions:
\begin{align}\label{eq:defaux}
W(x,y)&=\left< \tr \frac{1}{x-A} \tr \frac{1}{y-A} \right>
\\
G(x)&=\left<\tr \frac{1}{x-A} XX^\dagger \right>,\qquad
\bar G(x)=\left<\tr \frac{1}{x-A} X^\dagger X \right>
\\
H(x,y)&=\left<\tr \frac{1}{x-A}X\frac{1}{y-A}X^\dagger\right>
\label{eq:defauxend}
\end{align}
Just like $W(x)$, these correlation functions can be interpreted combinatorially. For now, all that we need to know
is that $G(x)$, $\bar G(x)$ (resp.\ $W(x,y)$, $H(x,y)$) are the weighted enumeration of maps with
one boundary (resp.\ two boundaries),
where the power of $N$ is its Euler--Poincar\'e characteristic $\chi$.
Note that in the two-boundary case, the surface is not necessarily connected;
in particular, the leading contribution at $N\to\infty$ corresponds to the topology of two 
disks, i.e., $W(x,y)\buildrel N\to\infty\over\sim N^2 \W(x)\W(y)$.

Expressing that the integral of a total derivative is zero, one has
\begin{align}\label{eq:loop1}
0&=Z_N^{-1}\int \d X\d X^\dagger\d A
\tr\frac{\der}{\der A} \left(\frac{1}{x-A}e^{NS}\right)
\\\notag
&=\left<\left(\tr \frac{1}{x-A}\right)^2-N t^{-1} \tr\frac{A}{x-A}+N\tr\frac{1}{x-A}(\omega XX^\dagger+\omega^{-1}X^\dagger X)
\right>
\intertext{Similarly,}
\label{eq:loop2}
0&=Z_N^{-1}\int \d X\d X^\dagger\d A
\tr\frac{\der}{\der X}\left(\frac{1}{x-A}X \frac{1}{y-A} e^{NS}\right)
\\\notag
&=\left<\tr\frac{1}{x-A}\tr\frac{1}{y-A}+N\tr\frac{1}{x-A}X \frac{1}{y-A}(-X^\dagger+\omega X^\dagger A+\omega^{-1}AX^\dagger)
\right>
\end{align}

In terms of $W(x)$ \eqref{eq:defW} and of the auxiliary correlation functions \eqref{eq:defaux}--\eqref{eq:defauxend},
we can rewrite \eqref{eq:loop1} and \eqref{eq:loop2}
\begin{align}\label{eq:loop1b}
0&=W(x,x)-N t^{-1} (x W(x)-N)+N(\omega G(x)+\omega^{-1}\bar G(x))
\\\label{eq:loop2b}
0&=W(x,y)+N H(x,y)(-1+\omega x+\omega^{-1}y)
-N \omega \bar G(y)-N\omega^{-1}G(x)
\end{align}

We now look at the leading behavior of \eqref{eq:loop1b} and \eqref{eq:loop2b} as $N\to\infty$.
Recalling that $W(x,y)\sim W(x)W(y)\sim N^2 \W(x)\W(y)$, we obtain equations \eqref{eq:loop1c} and \eqref{eq:loop2c}:
\begin{align*}
0&=\W(x)^2-t^{-1}(x \W(x)-1)+\omega \G(x)+\omega^{-1}\bG(x)
\\
0&=\W(x)\W(y)+\H(x,y)(-1+\omega x+\omega^{-1}y)-\omega\bG(y)-\omega^{-1}\bG(x)
\end{align*}

\section*{Acknowledgments}
We would like to thank Alin Bostan for pointing out several references for Proposition \ref{prop:classical_Dfinite}. The first author would also like to thank Mireille Bousquet-M\'elou for several useful conversations about the subject.

\gdef\MRshorten#1 #2MRend{#1}%
\gdef\MRfirsttwo#1#2{\if#1M%
MR\else MR#1#2\fi}
\def\MRfix#1{\MRshorten\MRfirsttwo#1 MRend}
\renewcommand\MR[1]{\relax\ifhmode\unskip\spacefactor3000 \space\fi
\MRhref{\MRfix{#1}}{{\scriptsize \MRfix{#1}}}}
\renewcommand{\MRhref}[2]{%
\href{http://www.ams.org/mathscinet-getitem?mr=#1}{#2}}
\bibliographystyle{amsalphahyper}
\bibliography{biblio}

\Addresses

\end{document}